\documentclass[a4paper,10pt]{article}

\usepackage{color}
\usepackage[english]{babel}
\usepackage{amsmath,amsfonts,amssymb,amsthm}
\usepackage{url}
\usepackage{graphicx}

\newtheorem{theorem}{Theorem}[section]

\newtheorem{lemma}{Lemma}[section]
\newtheorem{proposition}{Proposition}[section]

\theoremstyle{definition}
\newtheorem{remark}{Remark}[section]

\numberwithin{equation}{section}

\newcommand\blfootnote[1]{\begingroup\renewcommand\thefootnote{}\footnote{#1}\addtocounter{footnote}{-1}\endgroup}

\begin{document}

\title{
{\bf\Large Pairs of nodal solutions for a Minkowski-curvature boundary value problem \hspace{1.5cm} in a ball}}

\author{
\vspace{1mm}
\\
{\bf\large Alberto Boscaggin}
\vspace{1mm}\\
{\it\small Dipartimento di Matematica, Universit\`a di Torino}\\
{\it\small via Carlo Alberto 10}, {\it\small 10123 Torino, Italy}\\
{\it\small e-mail: alberto.boscaggin@unito.it}\vspace{1mm}\\
\vspace{1mm}\\
{\bf\large Maurizio Garrione}
\vspace{1mm}\\
{\it\small Dipartimento di Matematica e Applicazioni, Universit\`a di Milano-Bicocca}\\
{\it\small via Cozzi 55}, {\it\small 20125 Milano, Italy}\\
{\it\small e-mail: maurizio.garrione@unimib.it}\vspace{1mm}\\}

\date{}

\maketitle

\vspace{-2mm}

\begin{abstract}
By using a shooting technique, we prove that the quasilinear boundary value problem 
$$
\left\{
\begin{array}{ll}
\displaystyle \textnormal{div} \, \left( \frac{\nabla u}{\sqrt{1-\vert \nabla u \vert^2}}\right) +  \lambda q(\vert x \vert) \vert u \vert^{p-1} u = 0, & \quad x \in \mathcal{B}, \vspace{0.2 cm} \\
u = 0, & \quad x \in \partial \mathcal{B},
\end{array}
\right.
$$
where $\mathcal{B} \subset \mathbb{R}^N$ is a ball and $p > 1$, has more and more pairs of nodal solutions on growing of the parameter $\lambda > 0$. The radial Neumann problem and the periodic problem for the corresponding one-dimensional equation are considered, as well. 
\end{abstract}

\section{Introduction}

In this paper, we study existence and multiplicity of radial solutions of the quasilinear equation
\begin{equation}\label{equaz}
\textnormal{div} \, \left( \frac{\nabla u}{\sqrt{1-\vert \nabla u \vert^2}}\right) +  f(\vert x \vert,u) = 0,  \qquad x \in \mathcal{B}_R,
\end{equation}
where $\mathcal{B}_R \subset \mathbb{R}^N$ is the ball of radius $R$, for $N \geq 2$. 
\smallbreak
As well known, the above differential equation can be meant as a prescribed mean curvature equation in the Minkowski space
\cite{BarSim82,Fla79,Ger83}; in recent years, the solvability of the associated boundary value problems - even in the non-radial setting - has gained a lot of interest among researchers in the field of Nonlinear Analysis (see, for instance, 
\cite{BerJebMaw09,BerJebMaw14} and the references therein).  
\smallbreak
The model case for our investigation will be the Dirichlet boundary value problem
\begin{equation}\label{equazione0}
\left\{
\begin{array}{ll}
\displaystyle \textnormal{div} \, \left( \frac{\nabla u}{\sqrt{1-\vert \nabla u \vert^2}}\right) +  \lambda q(\vert x \vert) \vert u \vert^{p-1} u = 0, & \quad x \in \mathcal{B}_R, \vspace{0.2 cm} \\
u = 0, & \quad x \in \partial \mathcal{B}_R,
\end{array}
\right.
\end{equation}
where $q: [0,R] \to \mathbb{R}$ is a continuous function, $p > 1$ and $\lambda$ is a positive parameter.
As shown in \cite{BerJebTor13,CoeCorRiv14}, the role of $\lambda$ is crucial for the existence of \emph{positive} solutions of \eqref{equazione0}. In particular, it was proved therein that, when $q^+ \not\equiv 0$ and $\lambda > 0$ is large enough, two positive (and two negative) solutions appear, while in general nonexistence holds for $\lambda \to 0^+$ (see also \cite{CoeCorObeOma12} for a previous achievement in a one-dimensional setting). 
Later on, such a result has been extended to the genuine PDE setting in \cite{BerJebMaw14,CorObeOmaRiv13}.
\smallbreak
The aim of this paper is to show that more and more \emph{nodal} (i.e., sign-changing) 
solutions of \eqref{equazione0} appear as well, on growing of the parameter $\lambda$. More precisely, as a corollary of our main result, we can state the following theorem.

\begin{theorem}\label{mainthm}
Let $p > 1$ and let $q: [0, R] \to \mathbb{R}$ be a continuous function such that $q(r_0) > 0$ for some $r_0 \in [0, R]$. Then, for any integer $k \geq 1$, there exists $\lambda^*(k) > 0$ such that, for every $\lambda > \lambda^*(k)$, problem \eqref{equazione0} has at least $4k$ nodal radial solutions. 
\end{theorem}

We anticipate that the above solutions will be distinguished according to their number of nodal regions; we refer to Theorem
\ref{mainthm2} for a precise description, in a more general setting.
\smallbreak
The underlying idea of the high-multiplicity scheme in Theorem \ref{mainthm} can be traced back to a paper by Rabinowitz \cite{Rab74}, dealing however with a semilinear second order equation like $u'' + q(x)g(u) = 0$, with $g(u)u < 0$ for $\vert u \vert$ large. Recently, this pattern has been shown to emerge in several different situations, both for ordinary \cite{BosZan13} and partial differential equations \cite{LopRab16}, always requiring a \emph{sublinear} behavior for $g(u)$ at infinity, 
i.e., $\limsup_{\vert u \vert \to \infty}g(u)/u \leq 0$.
Theorem \ref{mainthm} can thus be seen as a further step in this line of research; it is remarkable, however, that the sublinearity of $g(u)$ at infinity is not needed, due to the peculiar properties of the Minkowski-curvature operator.
\smallbreak
As for the proof of Theorem \ref{mainthm}, we adopt a shooting approach for the equivalent ODE formulation
$$
\left\{
	\begin{array}{l}
		\displaystyle (r^{N-1} \varphi(u'))' + \lambda r^{N-1} q(r) \vert u \vert^{p-1} u = 0 \vspace{0.2 cm}\\
		u'(0)=0, \; u(R)=0, 
	\end{array}
	\right.
$$
that is, we consider the associated singular Cauchy problem 
$$
u'(0)=0, \quad u(0)=\eta,
$$
on varying of the initial datum $\eta \in \mathbb{R}$, looking for values $\eta \neq 0$ such that $u(R)=0$. A careful study shows that small 
(i.e., $0 < \vert \eta \vert \to 0$) and large (i.e., $\vert \eta \vert > R$) solutions 
do not rotate around the origin in the phase-plane $(u,r^{N-1}\varphi(u'))$, while intermediate ones rotate more and more on growing of $\lambda$. As a consequence, similarly as in \cite{BosZan13}, a ``double-gap'' picture for the winding number is created and multiple \emph{pairs} of solutions with precise nodal characterization can be provided.
\smallbreak
It is worth noticing that the above argument also allows us to easily recover the existence of four one-signed radial solutions 
(two positive ones and two negative ones) already proved in \cite{BerJebTor13,CoeCorRiv14} with topological and variational techniques, respectively, see Remark \ref{positive}. Incidentally, we mention that a shooting approach has been recently exploited to investigate the existence of radial ground-state solutions, as well \cite{Azz14,Azz16}.
\smallbreak
The plan of the paper is the following. In Section \ref{sez2} we present a preliminary technical result, providing solutions rotating around the origin for a general planar Hamiltonian system. In Section \ref{sez3} we state and prove our main result,
dealing with a nonlinearity of the type $f(\vert x \vert, u) =\lambda q(\vert x \vert)g(u)$, under minimal assumptions on $g(u)$. We then show some numerical simulations in order to ease the reader's comprehension of the statement. Finally, Section \ref{sez4} is devoted to several remarks about possible extensions and related results.
In particular, we deal with the periodic problem for the one-dimensional counterpart of \eqref{equaz}, namely
$$
\left( \frac{u'}{\sqrt{1-(u')^2}} \right)' +  f(t,u) = 0,  \qquad t \in \mathbb{R},
$$
often used to model the relativistic version of Newton's law (see \cite{Maw13} and its rich bibliography).

\section{An auxiliary result}\label{sez2}

In this section, we are going to present an auxiliary result dealing with a planar Hamiltonian system of the type
\begin{equation}\label{sistema}
	\left\{
	\begin{array}{l}
		\displaystyle x' = X(t, y) \vspace{0.1cm} \\
		\displaystyle y' = -Y(t, x),  
	\end{array}
	\right.
\end{equation}
where $X,Y: I \times \mathbb{R} \to \mathbb{R}$ are locally Lipschitz continuous functions.
Roughly speaking, we are going to prove that, whenever suitable (local) sign-conditions are assumed, the number of rotations around the origin of the planar path $(x(t),y(t))$ becomes arbitrarily large on growing of the width of the interval $I$.

To this end, we will write (whenever possible, namely for $x(t)^2 + y(t)^2 > 0$)
solutions of \eqref{sistema} in (clockwise) polar coordinates as
\begin{equation}\label{polari2}
x(t) = \rho(t) \cos \vartheta(t), \qquad y(t) = - \rho(t) \sin\vartheta(t),
\end{equation}
with $\rho(t) > 0$. Notice that, of course, the angular coordinate $\vartheta(t)$ is defined up to integer multiples of $2\pi$; however, the expression $\vartheta(t_2) - \vartheta(t_1)$, for any $t_1,t_2 \in I$, is uniquely determined, depending on the path $(x(t),y(t))$ only.

With this in mind, we state and prove the following result, which is a slight variant of \cite[Lemma 3.2]{BosZan13}.
 
\begin{proposition}\label{rotazioni}
Let $a_i, b_i: (-\delta, \delta) \to \mathbb{R}$, with $i=1,2$, be locally Lipschitz continuous functions such that 
\begin{equation}\label{controllogh}
0 < a_1(s) s \leq b_1(s) s \quad \textrm{ and } \quad 0 < b_2(s) s \leq a_2(s) s, \quad \textrm{ for every } s \neq 0.   
\end{equation}
Then, for every positive integer $j$, there exist $\tau_j^* > 0$ and $\rho_j^* \in (0, \delta)$ such that, for every interval $I=[t_0, t_1] \subset \mathbb{R}$ with $t_1 - t_0 > \tau_j^*$ and for every locally Lipschitz continuous functions $X, Y: I \times \mathbb{R} \to \mathbb{R}$ satisfying 
$$
a_1(y) y \leq X(t, y) y \leq b_1(y) y, \quad \textrm{ for every } t \in I, \,y \in (-\delta, \delta)
$$
and
$$
b_2(x) x \leq Y(t, x) x \leq a_2(x) x, \quad \textrm{ for every } t \in I, \,x \in (-\delta, \delta),
$$
it holds that any solution $(x(t), y(t))$ of \eqref{sistema}, defined on $I$ and satisfying  $x(t_0)^2 + y(t_0)^2 = (\rho_j^*)^2$, fulfills $x(t)^2 + y(t)^2 > 0$ for every $t \in I$ and  
$$
\vartheta(t_1) - \vartheta(t_0) > j\pi. 
$$
\end{proposition}

\begin{proof}[Sketch of the proof]
We are going to give a sketch of the proof, following the arguments in \cite[Lemma 3.2]{BosZan13}.
The crucial point is to construct two spiraling planar curves controlling (from below and from above) the rotations of the planar path $(x(t),y(t))$; this can be done be gluing together pieces of level curves of suitable energy functions.
Precisely, after having extended $a_i$ and $b_i$ to the whole real line in such a way that \eqref{controllogh} still holds and that the primitives
$$
A_i(s) = \int_0^s a_i(\xi)\,d\xi, \qquad B_i(s) = \int_0^s b_i(\xi)\,d\xi 
$$ 
are coercive, one defines the energies
$$
\mathcal{E}_A(x,y) = A_1(y) + A_2(x), \qquad \mathcal{E}_B(x,y) = B_1(y) + B_2(x).
$$
Straightforward computations show that, along a solution $(x(t),y(t))$ of \eqref{sistema}, it holds 
$$
\frac{d}{dt}\mathcal{E}_A(x(t),y(t)) \geq 0, \; \mbox{ if } x(t)y(t) \geq 0, \qquad
\frac{d}{dt}\mathcal{E}_A(x(t),y(t)) \leq 0, \; \mbox{ if } x(t)y(t) \leq 0
$$
and
$$
\frac{d}{dt}\mathcal{E}_B(x(t),y(t)) \geq 0, \; \mbox{ if } x(t)y(t) \leq 0, \qquad
\frac{d}{dt}\mathcal{E}_B(x(t),y(t)) \leq 0, \; \mbox{ if } x(t)y(t) \geq 0.
$$
Therefore, the aforementioned parameterized spirals 
$$
\mathbb{R} \ni \vartheta \mapsto (\rho_{\pm}(\vartheta)\cos \vartheta,- \rho_{\pm}(\vartheta)\sin\vartheta) \in \mathbb{R}^2 \setminus \{(0,0)\}
$$ 
can be obtained solving the differential equations
$$
\frac{d\rho_{\pm}}{d\vartheta} = \rho_{\pm} \mathcal{M}_{\pm} (\rho_{\pm}(\vartheta)\cos \vartheta,- \rho_{\pm}(\vartheta)\sin\vartheta),
$$
where
$$
\mathcal{M}_-(x,y) = 
\begin{cases}
\displaystyle \frac{b_1(y)x - b_2(x)y}{b_2(x)x+b_1(y)y}, & \textrm{if } \, xy \geq 0 \vspace{0.2cm}\\
\displaystyle \frac{a_1(y)x - a_2(x)y}{a_2(x)x+a_1(y)y}, & \textrm{if } \, xy \leq 0
\end{cases}
$$
and
$$
\mathcal{M}_+(x,y) =
\begin{cases}
\displaystyle \frac{b_1(y)x - b_2(x)y}{b_2(x)x+b_1(y)y}, & \textrm{if } \, xy \leq 0 \vspace{0.2cm}\\
\displaystyle \frac{a_1(y)x - a_2(x)y}{a_2(x)x+a_1(y)y}, & \textrm{if } \, xy \geq 0.
\end{cases}
$$
The remaining part of the proof follows exactly as in \cite[Lemma 3.2]{BosZan13}. 
\end{proof}

\section{Statement and proof of the main result}\label{sez3}

In this section, we state and prove our main result concerning the problem 
\begin{equation}\label{equazione}
\left\{
\begin{array}{ll}
\displaystyle \textnormal{div} \, \left( \frac{\nabla u}{\sqrt{1-\vert \nabla u \vert^2}}\right) + \lambda q(\vert x \vert) g(u) = 0, & \quad x \in \mathcal{B}_R, \vspace{0.2 cm} \\
u = 0, & \quad x \in \partial \mathcal{B}_R,
\end{array}
\right.
\end{equation}
where $\mathcal{B}_R \subset \mathbb{R}^N$ is the ball of radius $R$, for $N \geq 2$, and $\lambda > 0$.
\begin{theorem}\label{mainthm2}
Let $q: [0, R] \to \mathbb{R}$ be a continuous function such that $q(r_0) > 0$ for some $r_0 \in [0, R]$ and let $g: \mathbb{R} \to \mathbb{R}$ be a locally Lipschitz continuous function satisfying 
\begin{itemize}
\item[$(g_0)$] there exists $\delta > 0$ such that 
$$
g(u) u > 0 \quad \textrm{ for every } u \in (-\delta, \delta) \setminus \{0\}
$$
and 
$$
\lim_{u \to 0} \frac{g(u)}{u} = 0. 
$$
\end{itemize}
Then, for any integer $k \geq 1$, there exists $\lambda^*(k) > 0$ such that, for every $\lambda > \lambda^*(k)$, problem \eqref{equazione} has at least $4k$ sign-changing radial solutions. More precisely, for every integer $j=1, \ldots k$ there exist four radial solutions $u_{l, j}^-$, $u_{s, j}^-$, $u_{s, j}^+$, $u_{l, j}^+$ of \eqref{equazione} satisfying 
$$
u_{l, j}^-(0) < u_{s, j}^-(0) < 0 < u_{s, j}^+(0) < u_{l, j}^+(0) 
$$
and having exactly $j+1$ nodal domains. 
\end{theorem} 
Theorem \ref{mainthm} in the Introduction is a direct consequence of this statement for $g(u)=\vert u \vert^{p-1} u$, which satisfies assumption $(g_0)$ as long as $p > 1$. Notice however that in Theorem \ref{mainthm2} only a \emph{local} condition on $g(u)$ at zero is required. We also observe that the solutions found are distinguished through their nodal properties; precisely, for any $j=1, \ldots, k$, we find four solutions, two of them with positive value in the center of the ball (a small one $u_{s, j}^+$ and a large one $u_{l, j}^+$) and two with negative value therein (a small one $u_{s, j}^-$ and a large one $u_{l, j}^-$); see Section \ref{sez3b} for a more accurate discussion.

\subsection{Proof of Theorem \ref{mainthm2}}
Setting
$$
\varphi(s) = \frac{s}{\sqrt{1-s^2}}, \qquad s \in (-1,1),
$$
and writing $u(r)=u(\vert x \vert)$ with the usual abuse of notation, we convert the radial problem associated with \eqref{equazione} into 
\begin{equation}\label{ode}
	\left\{
	\begin{array}{l}
		\displaystyle (r^{N-1} \varphi(u'))' + \lambda r^{N-1} q(r) g(u) = 0 \vspace{0.2 cm}\\
		u'(0)=0, \; u(R)=0. 
	\end{array}
	\right.
\end{equation}
\begin{remark}\label{regolarita}
\textnormal{Let us recall that a solution of \eqref{ode} is meant as a function $u \in C^1([0, R])$ such that 
$\vert u'(r) \vert < 1$ for every $r \in [0,R]$, 
$r^{N-1} \varphi(u') \in C^1([0, R])$ and the differential equation in \eqref{ode} is satisfied pointwise, together with the boundary conditions. Actually, by an easy regularity argument (see \cite[Remark 3.3]{CoeCorRiv14}) one can see that $\varphi(u') \in C^1([0,R])$, 
finally implying $u \in C^2([0, R])$.}   
\end{remark}
As a first step, we are going to introduce an equivalent formulation of this problem, obtained from a suitable modification of both the differential operator and the nonlinear term (cf. \cite[Proposition 2.1]{CoeCorRiv14}). Precisely, on one hand we choose a locally Lipschitz continuous function $\tilde{f}: [0, R] \times \mathbb{R} \to \mathbb{R}$ such that
$$
\tilde{f}(r, u) = 
\left\{
\begin{array}{ll}
q(r) g(u) & \textrm{if } \vert u \vert \leq R \vspace{0.1cm}\\
0 & \textrm{if } \vert u \vert \geq R+1,
\end{array}
\right.
$$
and we set $M=\sup_{r \in [0, R], u \in \mathbb{R}} \vert \tilde{f}(r, u) \vert$. On the other hand, 
we set $\gamma=\varphi^{-1}(\lambda MR) \in (0, 1)$ and define the $C^1$-function $\tilde{\varphi}: \mathbb{R} \to \mathbb{R}$ as 
$$
\tilde{\varphi}(x)=
\left\{
\begin{array}{ll}
\varphi(x) & \textrm{if } \vert x \vert \leq \gamma \vspace{0.1cm}\\
\varphi'(\gamma) (x+\gamma) - \varphi(\gamma) & \textrm{if } x < -\gamma \vspace{0.1cm}\\
\varphi'(\gamma) (x-\gamma) + \varphi(\gamma) & \textrm{if } x > \gamma.
\end{array}
\right.
$$
We then introduce the modified problem
\begin{equation}\label{modificato}
	\left\{
	\begin{array}{l}
		\displaystyle (r^{N-1} \tilde{\varphi}(u'))' + \lambda r^{N-1} \tilde{f}(r, u) = 0 \vspace{0.2 cm}\\
		u'(0)=0, \; u(R)=0, 
	\end{array}
	\right.
\end{equation}
meaning its solutions as in Remark \ref{regolarita} (up to the requirement $\vert u'(r) \vert < 1$),
and state the following lemma.
\begin{lemma}\label{equivalenza}
Let $u \in C^1([0, R])$; then, $u(r)$ is a solution of \eqref{ode} if and only if $u(r)$ is a solution of \eqref{modificato}. 
\end{lemma}
\begin{proof}
Let $u(r)$ be a solution of \eqref{modificato}; integrating the equation and using $u'(0) = 0$, we obtain
$$
r^{N-1} \tilde\varphi(u'(r)) = - \lambda \int_0^r s^{N-1} \tilde f(s,u(s))\,ds, \quad \textrm{ for every } r \in [0,R],
$$ 
implying
\begin{equation}\label{u'}
\vert u'(r) \vert \leq \tilde\varphi^{-1} \left( \lambda M R \right) = \varphi^{-1} \left( \lambda M R \right) = \gamma < 1. 
\end{equation}
Hence, on one hand $\tilde\varphi(u'(r)) = \varphi(u'(r))$ for every $r \in [0,R]$; on the other hand,
using $u(R) = 0$ we have that 
\begin{equation}\label{u}
\vert u(r) \vert \leq \int_0^R \vert u'(s) \vert \, ds \leq \gamma R < R, \quad \textrm{ for every } r \in [0,R], 
\end{equation}
so that $\tilde f(r,u(r)) = q(r)g(u(r))$, as well. Summing up, $u(r)$ solves \eqref{ode}.

The converse (which will not be used for our purposes) follows from similar arguments and we omit the proof.
\end{proof}
According to Lemma \ref{equivalenza}, from now on we deal with problem \eqref{modificato}. Having in mind the shooting approach presented in the Introduction, we first state a global existence and continuous dependence result for the associated Cauchy problems.
\begin{lemma}\label{unicita}
For any $\eta \in \mathbb{R}$, the Cauchy problem 
\begin{equation}\label{PC}
	\left\{
	\begin{array}{l}
		\displaystyle (r^{N-1} \tilde{\varphi}(u'))' + \lambda r^{N-1} \tilde{f}(r, u) = 0 \vspace{0.2 cm}\\
		u'(0)=0, \; u(0)=\eta, 
	\end{array}
	\right.
\end{equation}
has a unique solution $u(r; \eta)$ defined on $[0, R]$. Moreover, if $\eta_k \to \bar{\eta}$, then 
\begin{equation}\label{convcon}
u(r; \eta_k) \to u(r; \bar{\eta}), \qquad r^{N-1} \tilde\varphi(u'(r; \eta_k)) \to r^{N-1}\tilde\varphi(u'(r; \bar{\eta})),
\end{equation}
uniformly on $[0, R]$.
\end{lemma} 

\begin{proof}
We first observe that $u(r;\eta)$ is a solution of \eqref{PC}, defined on $[0,R]$, if and only if it is a fixed point of the operator $T: C([0,R]) \to C([0,R])$ defined by
$$
Tu(r) = \eta - \int_0^r \tilde\varphi^{-1} \left( \frac{1}{s^{N-1}}\int_0^s \lambda \tau^{N-1}\tilde f(\tau,u(\tau))\,d\tau\right)\,ds;
$$ 
incidentally, we explicitly notice that $T$ is well-defined, despite the presence of the singular term $1/s^{N-1}$. 
We are going to establish existence and uniqueness of a fixed point by proving that a suitable iterate of $T$ is a contraction with respect to the standard sup-norm $\Vert u \Vert = \sup_{r \in [0,R]} \vert u(r) \vert$. 
Precisely, denoting by $\textnormal{Lip}(\tilde\varphi^{-1})$ and 
$\textnormal{Lip}(\tilde f)$ the Lipschitz constants of $\tilde\varphi^{-1}$ and $\tilde f$ respectively, and setting
$$
L = \frac{\lambda}{N}\, \textnormal{Lip}(\tilde\varphi^{-1}) \textnormal{Lip}(\tilde f),
$$
we show by induction that, for every $k \in \mathbb{N}$ and for every $u_1$, $u_2 \in C([0,R])$, it holds
\begin{equation}\label{indu}
\vert T^k u_1(r) - T^ku_2(r) \vert \leq \frac{L^k r^{2k}}{k!} \Vert u_1 - u_2 \Vert, \quad \mbox{ for every }
r \in [0,R],
\end{equation}
from which the thesis easily follows for $k$ large enough.
  
For $k = 0$, the estimate \eqref{indu} is trivial. Assuming it for $k \geq 0$, we have
\begin{align*}
\vert & T^{k+1} u_1(r) - T^{k+1}u_2(r) \vert \leq 
\int_0^r \left\vert 
\tilde\varphi^{-1}\left( \frac{1}{s^{N-1}}\int_0^s \lambda \tau^{N-1}\tilde f(\tau,T^k u_1(\tau))\,d\tau \right)\right. + \\
& \qquad \qquad  - \left.\tilde\varphi^{-1}\left( \frac{1}{s^{N-1}}\int_0^s \lambda \tau^{N-1}\tilde f(\tau,T^k u_2(\tau))\,d\tau \right)\right\vert \,ds \\
& \leq \lambda \textnormal{Lip}(\tilde\varphi^{-1}) \textnormal{Lip}(\tilde f)
\int_0^r \frac{1}{s^{N-1}}\int_0^s \tau^{N-1}\vert T^k u_1(\tau) - T^k u_2(\tau) \vert\,d\tau \,ds \\
& \leq \frac{\lambda^{k+1} \textnormal{Lip}(\tilde\varphi^{-1})^{k+1} \textnormal{Lip}(\tilde f)^{k+1}}{N^k k!}
\left(\int_0^r \frac{1}{s^{N-1}}\int_0^s \tau^{N+2k-1}\,d\tau \,ds \right) \Vert u_1 - u_2 \Vert \\
& = \frac{\lambda^{k+1} \textnormal{Lip}(\tilde\varphi^{-1})^{k+1} \textnormal{Lip}(\tilde f)^{k+1} r^{2k+2}}{N^k (N+2k) k! (2k+2)} \Vert u_1 - u_2 \Vert \leq \frac{L^{k+1} r^{2(k+1)}}{(k+1)!} \Vert u_1 - u_2 \Vert,
\end{align*}
thus implying that \eqref{indu} holds true.

The first convergence in \eqref{convcon} is a direct consequence of the above argument, while the second one now follows directly by integrating the differential equation in \eqref{PC}.
\end{proof}

Henceforth, we set 
$$
v(r; \eta) = r^{N-1} \tilde{\varphi}(u'(r; \eta)) 
$$
and we pass to (clockwise) polar-like coordinates, by writing 
$$
u(r; \eta)= \rho(r; \eta) \cos \theta(r; \eta), \quad v(r; \eta) = - \sqrt{\lambda}\rho(r; \eta) \sin \theta(r; \eta),
$$
where $\theta(0; \eta) = 0$ for $\eta > 0$ and $\theta(0; \eta) = \pi$ for $\eta < 0$.  
Notice that this change of variables is admissible for $\eta \neq 0$, since 
$$
u(r; \eta)^2 + v(r; \eta)^2 > 0 \quad \textrm{ for every } r \in [0, R], 
$$
this being a consequence of the uniqueness of the Cauchy problems (indeed, on $(0, R]$ the differential equation in \eqref{modificato} can be written as a first order planar system satisfying the assumption of Cauchy-Lipschitz theorem). 
Moreover, by standard results on path liftings, the continuity of the path $(u(r; \eta), v(r; \eta))$ with respect to $\eta$ ensures that $\theta(r; \eta)$ depends continuously on $\eta \neq 0$, as well. 

We also highlight the following property of the function $\theta(r; \eta)$, which will be used in the proofs. 
\begin{lemma}\label{primaprop}
For every $r_1, r_2$ such that $0 \leq r_1 \leq r_2 \leq R$, it holds that 
$$
\theta(r_2; \eta) - \theta(r_1; \eta) > -\pi. 
$$
\end{lemma}
\begin{proof}
The result follows easily from the fact that 
\begin{align*}
\theta'(r; \eta) & = \sqrt{\lambda}\,\frac{ u'(r; \eta) v(r; \eta)- v'(r; \eta) u(r; \eta)}{\lambda u(r; \eta)^2 + v(r; \eta)^2} \\
& = \sqrt{\lambda}\,\frac{ r^{N-1} \tilde{\varphi}(u'(r; \eta)) u'(r; \eta) - v'(r; \eta) u(r; \eta)}{\lambda u(r; \eta)^2 + v(r; \eta)^2} ,
\end{align*}
whence $\theta'(r; \eta) \geq 0$ if $\theta(r; \eta) = \pi/2 + k\pi$, namely if $u(r; \eta) = 0$ and $v(r; \eta) \neq 0$ (see also \cite[Lemma 3.1]{BosZan13}).   
\end{proof}
From now on, we consider the case when $\eta > 0$ and show how to find the $2k$ solutions $u_{s, j}^+, u_{l, j}^+$ for $j=1, \ldots, k$. 
A symmetric argument yields the other pairs of solutions. 
\smallbreak
According to the general strategy described in the Introduction, we first focus on the behavior of the small solutions; without loss of generality, we take the constant $\delta$ appearing in assumption $(g_0)$ strictly less than $R$.    
\begin{lemma}\label{piccole}
For every $\lambda > 0$, there exists $\eta_*(\lambda) \in (0, \delta)$ such that, for any $0 < \eta \leq \eta_*(\lambda)$, it holds that 
$$
\theta(r; \eta) < \frac{1}{2}\pi, \quad \textrm{ for every } r \in [0, R]. 
$$
\end{lemma}
\begin{proof}
Let us fix $\lambda > 0$ and choose $\varepsilon > 0$ such that 
\begin{equation}\label{boundeps}
\sqrt{\varepsilon} < \frac{\pi}{2R}. 
\end{equation}
By assumption $(g_0)$, there exists $\hat{\eta}(\lambda) \in (0, \delta)$ such that 
\begin{equation}\label{assf}
\lambda \tilde{f}(r, u) u = \lambda q(r) g(u) u \leq \varepsilon u^2, \quad \textrm{ for every } r \in [0, R], \; \vert u \vert \leq \hat{\eta}(\lambda).  
\end{equation}
By Lemma \ref{unicita}, we finally find $\eta_*(\lambda) \in (0, \hat{\eta}(\lambda))$ such that, if $\eta \in [0, \eta_*(\lambda)]$, it holds 
\begin{equation}\label{controllou}
\vert u(r; \eta) \vert \leq \hat{\eta}(\lambda), \quad \textrm{ for every } r \in [0, R].
\end{equation}
With these positions, we are going to prove that the planar path $(u(r; \eta), v(r; \eta))$, for $r \in [0, R]$, cannot reach the negative $v$-semiaxis, which clearly implies the conclusion. 

By contradiction, suppose that this is not the case. Then, there exists $[r_1, r_2] \subset [0, R]$ such that the path $r \mapsto (u(r), v(r))=(u(r; \eta), v(r; \eta))$ satisfies 
\begin{equation}\label{quartogiro}
v(r_1)=0=u(r_2), \quad v(r) < 0 < u(r) \textrm{ for every } r \in (r_1, r_2). 
\end{equation}
Of course, also the path 
$$
r \mapsto \left(\sqrt{\varepsilon} u(r), \tilde{\varphi}(u'(r))=\tfrac{v(r)}{r^{N-1}}\right)
$$
satisfies \eqref{quartogiro}; as a consequence, passing to clockwise polar coordinates 
as in \eqref{polari2}, we find 
$$
\frac{1}{4} = \frac{\sqrt{\varepsilon}}{2\pi} \int_{r_1}^{r_2} \frac{u'(r) \tilde{\varphi}(u'(r))- u(r) (\tilde{\varphi}(u'(r)))'}{\varepsilon u(r)^2 + \tilde{\varphi}(u'(r))^2} \, dr, 
$$
that is, in view of the differential equation in \eqref{modificato}, 
$$
\frac{1}{4} = \frac{\sqrt{\varepsilon}}{2\pi} \int_{r_1}^{r_2} \frac{u'(r) \tilde{\varphi}(u'(r))  + \lambda \tilde{f}(r, u(r)) u(r) + \tfrac{N-1}{r} \tilde{\varphi}(u'(r)) u(r)}{\varepsilon u(r)^2 + \tilde{\varphi}(u'(r))^2} \, dr. 
$$
Recalling that $\tilde{\varphi}(u'(r)) u(r)  \leq 0$ for every $r \in [r_1, r_2]$ and using \eqref{assf}-\eqref{controllou}, together with the inequality $s \tilde{\varphi}(s) \leq \tilde{\varphi}(s)^2$, we finally obtain 
$$
\frac{1}{4} < \frac{\sqrt{\varepsilon}R}{2\pi}, 
$$
contradicting \eqref{boundeps}. 
\end{proof}
Second, we prove our key lemma; roughly speaking, we make larger solutions rotate as desired by enlarging the parameter $\lambda$.  
\begin{lemma}\label{cruciale}
For every integer $k \geq 1$, there exists $\lambda_k^* > 0$ such that for every $\lambda > \lambda_k^*$ there exists $\eta^*(\lambda) \in (0,R+1)$ such that
$$
\theta(R; \eta^*(\lambda)) > (k + 1)\pi.  
$$
\end{lemma}

\begin{proof}
We first choose $0 < r_0^- < r_0^+ < R$ such that $q(r) \geq m > 0$ for $r \in [r_0^-, r_0^+]$ and  state the following. 
\smallbreak
\noindent
\underline{Claim}. \emph{There exist $\lambda_k^*> 0$ and $\Gamma_k \in (0, \delta)$ such that, for every $\lambda > \lambda_k^*$, every solution $(u(r), v(r)) =(\rho(r)\cos\theta(r), -\sqrt{\lambda}\rho(r)\sin\theta(r))$ of \eqref{modificato} satisfying $u(r_0^-)^2 + \tfrac{1}{\lambda} v(r_0^-)^2 = \Gamma_k^2$ fulfills 
$$
\theta(r_0^+) - \theta(r_0^-)  > (k + 3)\pi.  
$$
}
\\
\emph{Proof of the claim}. We set 
$$
(x(t), y(t))=\left(u\left(\frac{t}{\sqrt{\lambda}}\right), \frac{1}{\sqrt{\lambda}} v\left(\frac{t}{\sqrt{\lambda}}\right)\right),
$$
for $t \in I_\lambda=[\sqrt{\lambda} r_0^-, \sqrt{\lambda} r_0^+]$. 
It is easily seen that $(x(t), y(t))$ solves the differential system 
$$
\left\{
\begin{array}{l}
\displaystyle x'=\frac{1}{\sqrt{\lambda}}\tilde{\varphi}^{-1}\left(\frac{\lambda^{(N/2)} y}{t^{N-1}}\right)=: X_\lambda(t, y) \vspace{0.1cm}\\
\displaystyle y'= - \frac{t^{N-1}}{\lambda^{(N-1)/2}} \tilde{f}\left(\frac{t}{\sqrt{\lambda}}, x\right)=: -Y_\lambda(t, x).
\end{array}
\right.
$$
In view of the choice of $r_0^-, r_0^+$, we have that, for every $t \in I_\lambda$ and $x \in (-\delta,\delta)$, it holds that
$$
m (r_0^-)^{N-1} g(x) x \leq Y_\lambda(t, x) x \leq \Vert q^+ \Vert_{L^\infty(0, R)} (r_0^+)^{N-1} g(x) x. 
$$ 
On the other hand, since 
$$
\frac{1}{\varphi'(\gamma)} s^2 \leq \tilde{\varphi}^{-1}(s) s \leq s^2, \quad \textrm{ for every } s \in \mathbb{R},
$$
we find that, for every $t \in I_\lambda$ and $y \in (-\delta,\delta)$, 
$$
\frac{1}{\varphi'(\gamma) (r_0^+)^{N-1}} y^2 \leq X_\lambda(t, y) y \leq \frac{y^2}{(r_0^-)^{N-1}}. 
$$
Taking into account that, denoting by $\vartheta(t)$ the angular coordinate associated with the planar path $(x(t), y(t))$ as in \eqref{polari2}, it holds that 
$$
\vartheta(\sqrt{\lambda}r_0^+)-\vartheta(\sqrt{\lambda}r_0^-) = \theta(r_0^+)-\theta(r_0^-), 
$$
Proposition \ref{rotazioni} can be applied with the choice $j=k+3$, providing
$$
\lambda_k^* = \left(\frac{\tau^*_{k+3}}{r_0^+ - r_0^-}\right)^2, \quad \Gamma_k=\rho^*_{k+3}. 
$$
\qed
\smallbreak
\noindent
We can now easily draw the conclusion: indeed, it holds that $u(r; 0) \equiv 0$ and $u(r; R+1) \equiv R+1$, since $\tilde{f}(r, u) = 0$ for $\vert u \vert \geq R+1$. Accordingly, recalling that $\Gamma_k < \delta < R$, there exists $\eta^*(\lambda) \in (0, R+1)$ such that 
\begin{equation}\label{ellisse}
u(r_0^-; \eta^*(\lambda))^2 + \frac{1}{\lambda} v(r_0^-; \eta^*(\lambda))^2 = \Gamma_k^2. 
\end{equation}
Since 
\begin{align*}
\theta(R; \eta^*(\lambda))  & = \theta(r_0^-; \eta^*(\lambda)) - \theta(0; \eta^*(\lambda)) + \theta(r_0^+; \eta^*(\lambda)) - \theta(r_0^-; \eta^*(\lambda)) \\
& \quad + \theta(R; \eta^*(\lambda)) - \theta(r_0^+; \eta^*(\lambda)),
\end{align*}
from the previous Claim and Lemma \ref{primaprop} the statement follows. 
\end{proof}

We are now in a position to conclude. Given $k \geq 1$, we fix $\lambda >  \lambda_k^*$ (where $\lambda_k^*$ is given by Lemma \ref{cruciale}) and  we consider the continuous function 
$$
(0, +\infty) \ni \eta \mapsto \theta(R; \eta) \in \mathbb{R}. 
$$
Of course, $\theta(R;R+1) = 0$; moreover, from Lemmas \ref{piccole} and \ref{cruciale} we infer that $\eta_*(\lambda) < \eta^*(\lambda)$ and that 
$$
\theta(R;\eta_*(\lambda)) < \frac{\pi}{2} < (k+1) \pi < \theta(R;\eta^*(\lambda)). 
$$
Then, the Intermediate Value theorem gives, for any integer $j = 1, \ldots, k$, two positive values $\eta_{s,j}^+, \eta_{l, j}^+$, with 
$$
\eta_*(\lambda) < \eta_{s,j}^+ < \eta^*(\lambda) < \eta_{l, j}^+,
$$ 
such that 
\begin{equation}\label{rotinfo}
\theta(R; \eta_{s,j}^+)= \theta(R; \eta_{l,j}^+) = \left(\frac{1}{2} + j\right)\pi. 
\end{equation}
The solutions 
$$
u_{s,j}^+(r) = u(r;\eta_{s,j}^+), \qquad u_{l,j}^+(r) = u(r;\eta_{l,j}^+) 
$$
are therefore distinct solutions of \eqref{ode}; moreover, from \eqref{rotinfo}, together with the fact that $\theta'(r) \geq 0$ whenever $u(r) = 0$  (compare with the proof of Lemma \ref{primaprop}), it follows that $u_{s,j}^+(r)$ and $u_{l,j}^+(r)$ have exactly $j$ zeros on $[0,R)$ (see also \cite[Lemma 3.8]{BosZan13}).

A similar argument works when $\eta < 0$, giving the other pair of solutions.  

\subsection{Some numerical simulations} \label{sez3b}

To give a better insight into the statement of Theorem \ref{mainthm}, we collect here below some numerical simulations obtained with MAPLE$^{\copyright}$ software (see Figures 1, 2 and 3 below) for the equation
\begin{equation}\label{equazexp}
\textnormal{div}\, \left(\frac{\nabla u}{\sqrt{1-\vert \nabla u \vert^2}}\right) + \lambda u^3 = 0, \qquad x \in \mathcal{B}_R,
\end{equation}
choosing $N=2$ and $R=10$. The Dirichlet solutions shown are found for $\lambda = 5$; for completeness, we also depict the one-signed solutions already found in \cite{BerJebTor13, CoeCorRiv14} (see Remark \ref{positive}).  
\begin{figure}[!h]
\centering
\includegraphics[scale=0.25]{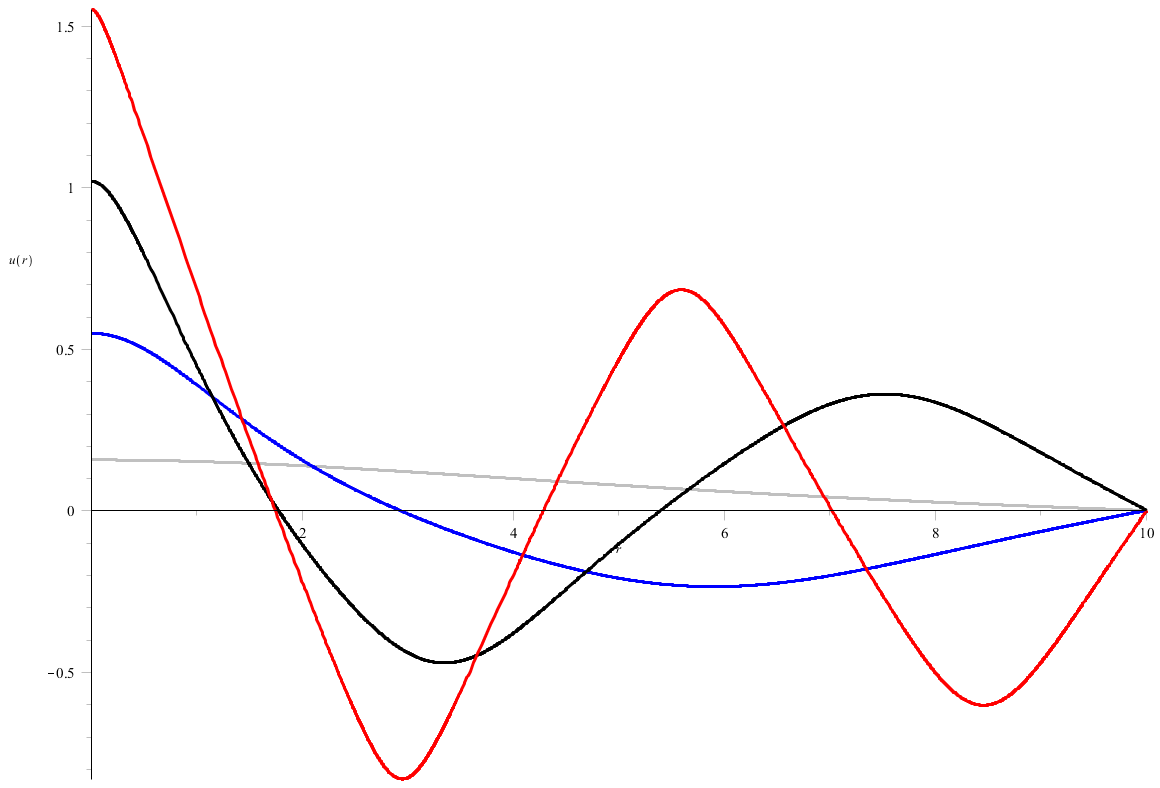}
\label{fig1}
\begin{flushleft}
\caption{We plot the small solutions $u_{s, j}^+(r)$ of equation \eqref{equazexp} for $j=0, 1, 2, 3$ (in grey, blue, black, red, respectively). Notice that we also plotted the positive solution $u_{s, 0}^+$, see Remark \ref{positive} in Section \ref{sez4}. Due to the oddness of the nonlinearity, here $u_{s, j}^-(r)=-u_{s, j}^+(r)$ for every $j=0, 1, 2, 3$.}
\end{flushleft}
\end{figure}

\begin{figure}[!h]
\centering
\includegraphics[scale=0.25]{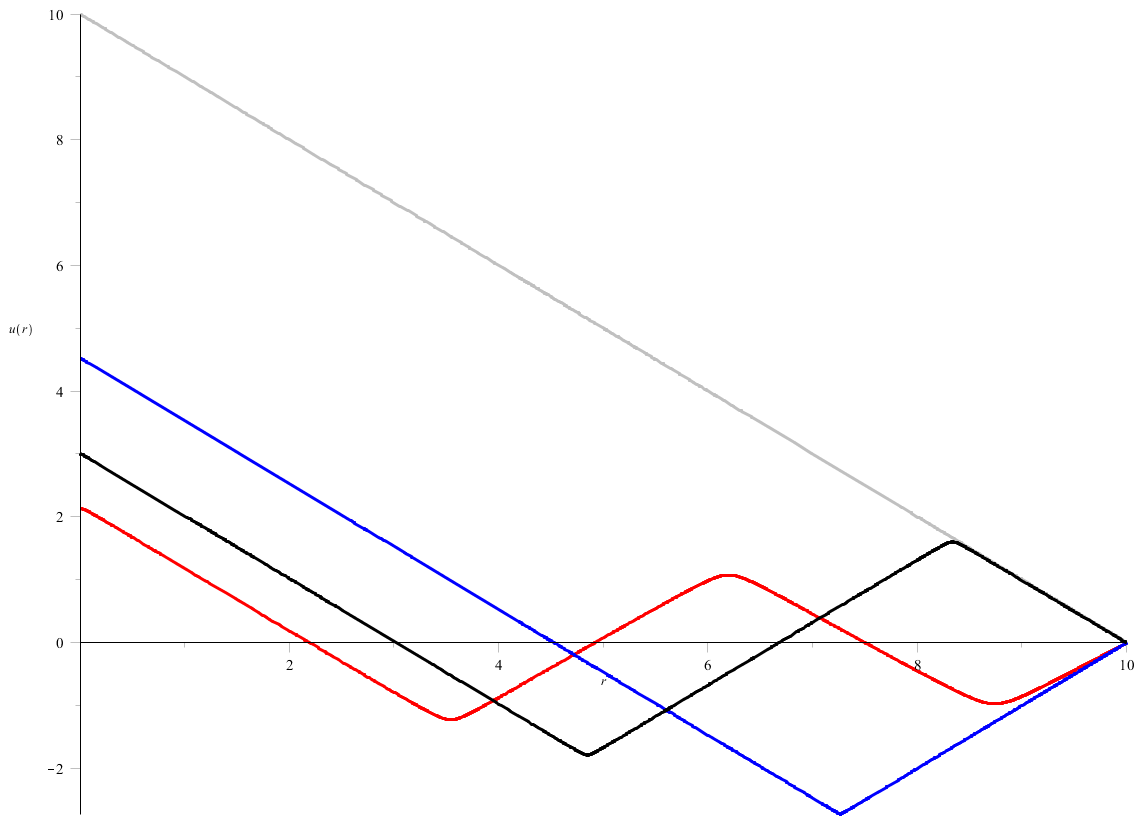}
\label{fig2}
\begin{flushleft}
\caption{We plot the large solutions $u_{l, j}^+(r)$ of equation \eqref{equazexp} for $j=0, 1, 2, 3$ (in grey, blue, black, red, respectively). Again, we plotted the positive solution $u_{l, 0}^+$, as well, and again $u_{l, j}^-(r)=-u_{l, j}^+(r)$.}
\end{flushleft}
\end{figure}

\begin{figure}[!h]
\centering
\includegraphics[scale=0.15]{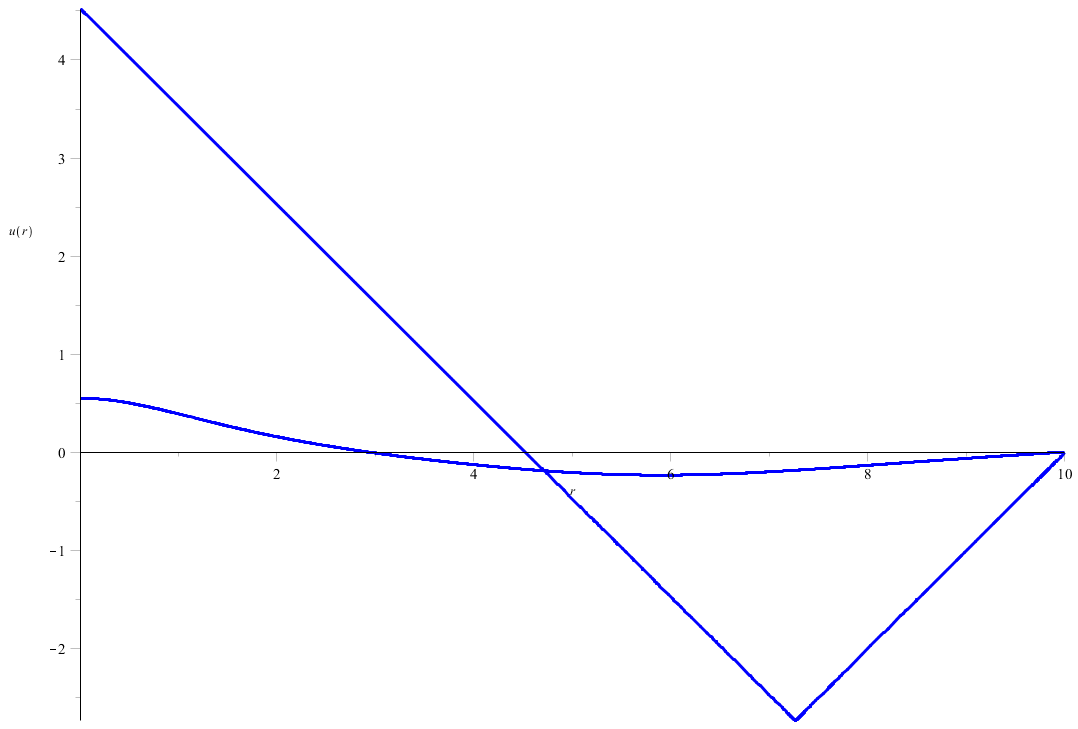}
\quad 
\includegraphics[scale=0.15]{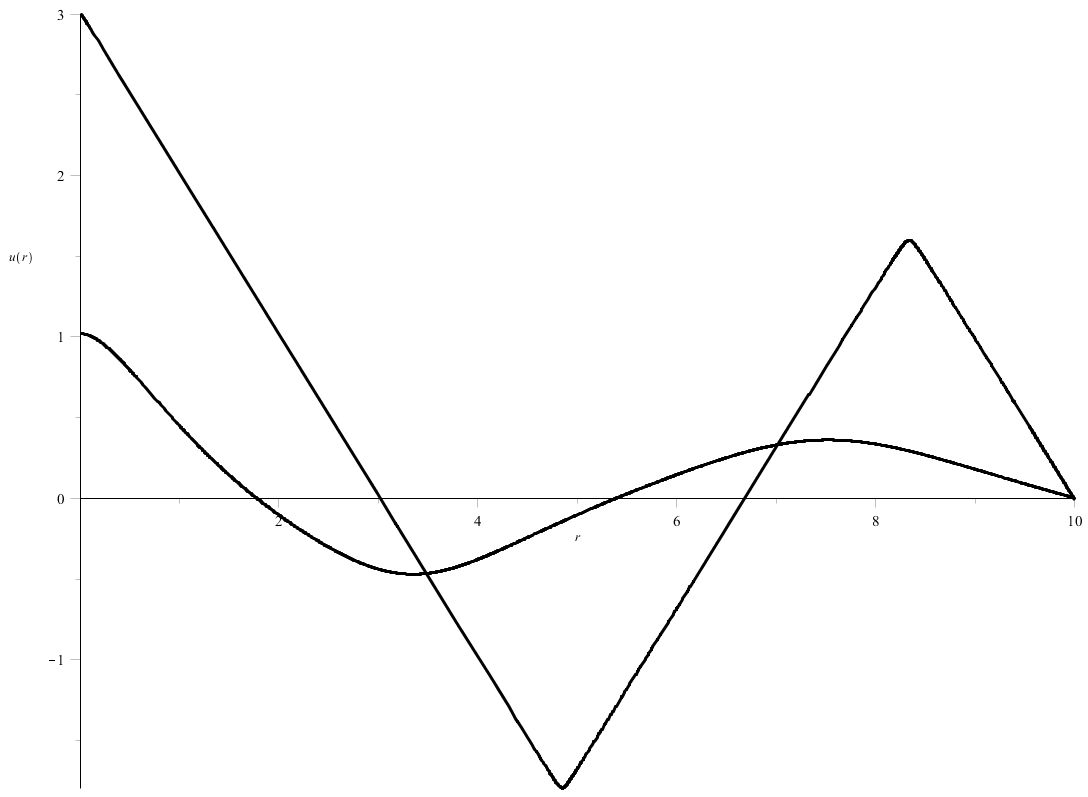}
\quad 
\medbreak
\includegraphics[scale=0.15]{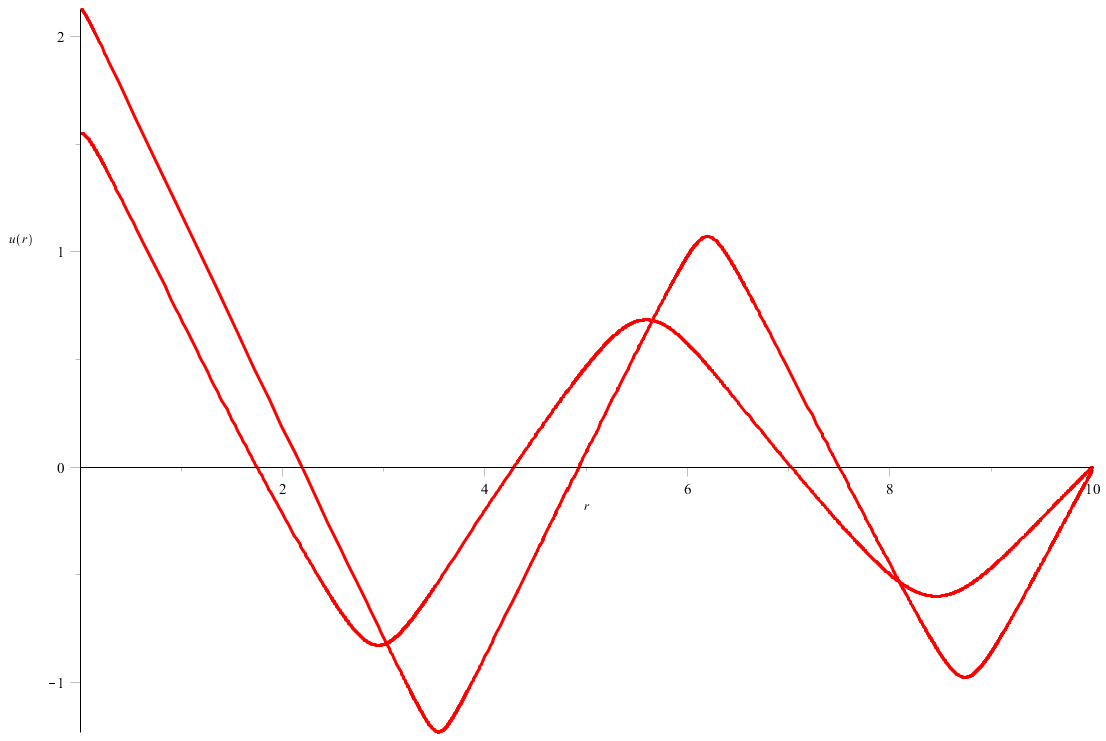}
\label{figura}
\begin{flushleft}
\caption{To better highlight the ``multiple pairs'' scheme, we group together the above displayed smaller and larger solutions of \eqref{equazexp} having the same nodal properties: namely, $u_{s, 1}^+$ and $u_{l, 1}^+$ in blue, $u_{s, 2}^+$ and $u_{l, 2}^+$ in black, $u_{s, 3}^+$ and $u_{l, 3}^+$ in red.}
\end{flushleft}
\end{figure}

The reader will certainly notice how the large solutions $u_{l, j}$ found above appear almost sharp-cornered, with slope approximately equal to $\pm 1$; we remark that, in principle, this is not a drawback of numerics but rather a consequence of the peculiar properties of the Minkowski curvature operator, as an elementary singular perturbation analysis shows. For simplicity, we assume that $q(r) > 0$ for every $r \in [0, R]$ and $g(u) u > 0$ for every $u \neq 0$ and we discuss the case when $\{u_\lambda\}_\lambda$ is a family of \emph{positive} solutions of \eqref{equazione}. Passing to the radial formulation, it is immediate to see that $u_\lambda$ is decreasing; moreover, using the arguments in the proof of Lemma \ref{equivalenza}, we obtain that $\{u_\lambda\}$ is bounded in the $C^1$-norm. By the sequential version of the Banach-Alaoglu theorem, there exists $u_\infty \in W^{1, \infty}$ such that, up to subsequences, $u_\lambda \to u_\infty$ in the weak$^*$-topology of $W^{1, \infty}$. Notice that $u_\infty$ is nonnegative and nonincreasing. Of course, it may be that $u_\infty \equiv 0$ (indeed, this is the case for the family of small positive solutions); however, assume that $u_\infty(r) > 0$ on $[0, \bar{r})$, for some $\bar{r} \in (0, R]$. Setting $v_\lambda(r)=r^{N-1} \varphi(u_\lambda'(r))$ and recalling that $v_\lambda(0)=0$, we have 
$$
v_\lambda(r) = -\lambda \int_0^r q(s) g(u_\lambda(s)) \, ds, \quad \textrm{ for every } r \in (0, R].
$$  
By Fatou Lemma, $v_\lambda(r) \to -\infty$ for any $r \in (0, R]$, so that $u_\lambda'(r) \to -1$ for any $r \in (0, R]$. By Lebesgue theorem, $u_\lambda' \to -1$ strongly in $L^p$ for every $p \geq 1$, so that $u_\infty(r)=R-r$ (notice that, a posteriori, $\bar{r}=R$) and $u_\lambda \to u_\infty$ strongly in $W^{1, p}$ for every $p \geq 1$. Notice that this convergence is sharp: indeed, since $u_\lambda \in C^1$ and $u_\lambda'(0) = 0$, if the convergence were $W^{1, \infty}$ it should be $u_\infty'(0)=0$ (while $u_\infty'(0) = -1$). 
The same argument works for negative solutions and, with some more care, is extendable to nodal solutions as well, showing that the graph of the singular limit $u_\infty(r)$ is a polygonal line with slopes $-1$ and $1$.

\section{Final remarks}\label{sez4}

We conclude the paper with some final remarks about possible extensions of Theorem \ref{mainthm}
and related results.

\begin{remark}\label{positive}
\textbf{One-signed solutions}. The choice $j=0$ in \eqref{rotinfo} is admissible, as well.
In this way, the existence of four one-signed (two positive and two negative) radial solutions 
 directly follows (compare with \cite{BerJebTor13,CoeCorRiv14}).
\end{remark}

\begin{remark}\label{neumann}
\textbf{The Neumann problem}. With similar arguments, we can deal as well with the Neumann problem
$$
\left\{
\begin{array}{ll}
\displaystyle \textnormal{div} \, \left( \frac{\nabla u}{\sqrt{1-\vert \nabla u \vert^2}}\right) + \lambda q(\vert x \vert) g(u) = 0, & \quad x \in \mathcal{B}_R, \vspace{0.2 cm} \\
\partial_\nu u = 0, & \quad x \in \partial \mathcal{B}_R.
\end{array}
\right.
$$
Indeed, when passing to the ODE formulation, one is led to the Neumann problem
$$
\left\{
	\begin{array}{l}
		\displaystyle (r^{N-1} \varphi(u'))' + \lambda r^{N-1} q(r) g(u) = 0 \vspace{0.2 cm}\\
		u'(0)=0, \; u'(R)=0
	\end{array}
	\right.
$$
and it is possible to find solutions by imposing, instead of \eqref{rotinfo}, the condition $\theta(R) = j \pi$, for any integer $j=1,\ldots,k$. It has to be noticed, however, that an analogue of Lemma \ref{equivalenza} can still be established only due to the fact that we deal with nodal solutions (ensuring the validity of \eqref{u}). Incidentally, we observe that in fact one-signed solutions of the Neumann problem typically do not exist (choose $q(r) \geq 0$ and integrate the equation).
\end{remark}

\begin{remark}\label{anelli}
\textbf{Solutions on annuli}. The same result as in Theorem \ref{mainthm2} holds true if the ball $\mathcal{B}_R$ is replaced by the annular domain $\mathcal{A}(R_1,R_2) = \{ x \in \mathbb{R}^N\, : \, R_1 < \vert x \vert < R_2 \}$. Here, the radial formulation reads as the Dirichlet problem
\begin{equation}\label{dir}
\left\{
	\begin{array}{l}
		\displaystyle (r^{N-1} \varphi(u'))' + \lambda r^{N-1} q(r) g(u) = 0 \vspace{0.2 cm}\\
		u(R_1)=0, \; u(R_2)=0
	\end{array}
	\right.
\end{equation}
and an analogue of Lemma \ref{equivalenza} holds true also in this case, since both $u(r)$ and $u'(r)$ vanish at least once in
$[R_1,R_2]$ (ensuring the validity of formulas \eqref{u'} and \eqref{u}). One can thus set, for the truncated equivalent problem, the shooting scheme
$$
u(R_1) = 0, \qquad u'(R_1) = \eta
$$
and look for solutions satisfying $\theta(R_2;\eta) - \theta(R_1;\eta) = j \pi$, for $j=1,\ldots,k$. The previous procedure of proof can still be followed, but unexpectedly some extra work is needed. Summarizing the main steps, it should be proved that:
\begin{itemize}
\item[-] $\theta(R_2;\eta) - \theta(R_1;\eta) < \frac{\pi}{2}$ for $0 < \vert \eta \vert$ small enough, as can be shown using arguments on the lines of the ones in Lemma \ref{piccole};
\item[-] $\theta(R_2;\eta) - \theta(R_1;\eta) < \frac{\pi}{2}$ for $\vert \eta \vert$ large enough, this being a consequence of the so-called elastic property (i.e., $u^2(r;\eta) + v^2(r;\eta) \to \infty$ uniformly when $\vert \eta \vert \to \infty$,
see \cite[Lemma 2]{Zan96} and notice that the differential equation in \eqref{dir} is not anymore singular) and of the fact that $\tilde f(r,u) \equiv 0$ for $\vert u \vert \geq R+1$;
\item[-] $\theta(R_2;\eta^*(\lambda)) - \theta(R_1;\eta^*(\lambda)) > k \pi$ for $\lambda > \lambda^*_k$, this being provable as in Lemma \ref{cruciale} using the elastic property once more to prove that \eqref{ellisse} holds true.
\end{itemize}
We refer the reader to \cite{BosZan13} for a detailed proof following this scheme, even if in a different setting. 
Combining with the previous Remark \ref{neumann}, one deduces the validity of Theorem \ref{mainthm2} for the Neumann problem on an annulus, as well.
\end{remark}

\begin{remark}
\textbf{The periodic problem}. When $q: \mathbb{R} \to \mathbb{R}$ is a continuous and $T$-periodic function ($T > 0$), a version of Theorem \ref{mainthm2} can also be stated for the $T$-periodic problem associated with the differential equation 
$$
\left(\frac{u'}{\sqrt{1-(u')^2}} \right)' + \lambda q(t) g(u) = 0.  
$$
In this case, instead of using the above shooting procedure, one has to apply the Poincar\'e-Birkhoff fixed point theorem. We briefly illustrate here below the steps of the proof, referring again the reader to \cite{BosZan13}:
\begin{itemize}
\item[-] we pass to the truncated problem $(\tilde\varphi(u'))' + \lambda \tilde{f}(t, u) = 0$; the analogue of Lemma \ref{equivalenza} holds true since both $u(t)$ and $u'(t)$ vanish at least once in $[0, T]$ (see the corresponding discussions in Remarks \ref{neumann} and \ref{anelli});
\item[-] we write the truncated equation as the planar Hamiltonian system 
$$
u' = \tilde\varphi^{-1}(v), \qquad v'=-\lambda \tilde{f}(t, u),
$$ 
and we denote by $(u(t; u_0, v_0), v(t; u_0, v_0))$ the solution satisfying the initial conditions $u(0)=u_0, v(0)=v_0$; 
\item[-] we pass to polar-like coordinates 
$$
\left\{
\begin{array}{l}
u(t; u_0, v_0) = \rho(t; u_0, v_0) \cos \theta(t; u_0, v_0) \vspace{0.1cm}\\
v(t; u_0, v_0) = - \sqrt{\lambda}\rho(t; u_0, v_0) \sin \theta(t; u_0, v_0);
\end{array}
\right.
$$
\item[-] we prove that $\theta(T; u_0, v_0) - \theta(0, u_0, v_0) < 2\pi$ both when $0 < u_0^2 + v_0^2$ is small enough and when it is large enough. This can be 
done again similarly as in the corresponding steps of Remark \ref{anelli};
\item[-] we prove that there exists $\lambda^*_k$ such that, for every $\lambda > \lambda^*_k$, there exists $\eta^*(\lambda)$ with $\theta(T; u_0, v_0) - \theta(0, u_0, v_0) > 2k\pi$ if $u_0^2 + v_0^2 = \eta^*(\lambda)^2$. This can be 
shown as in Lemma \ref{cruciale} (without loss of generality, we can assume $q(0) > 0$ so that the argument therein can be made slightly simpler). 
\end{itemize}
The conclusion then follows from the Poincar\'e-Birkhoff fixed point theorem. We also mention that, combining the arguments in the present paper with the ones in \cite[Section 3]{CorOmaZan16} it should be possible to prove the existence of pairs of \emph{subharmonic} solutions (namely, $mT$-periodic solutions for $m > 1$) for the equation 
$$
\left(\frac{u'}{\sqrt{1-(u')^2}} \right)' + q(t) g(u) = 0.  
$$
In this case, the largeness of $\lambda$ is replaced by the width of the periodicity interval (this requiring, however, that $q(t) > 0$ for every $t$). For the sake of briefness, we omit the details. 
\end{remark}
\begin{remark}
\textbf{Some complementary situations}. We finally observe that the superlinearity assumption in $(g_0)$, namely $\lim_{u \to 0} g(u)/u = 0$, is used only to ensure the validity of Lemma \ref{piccole}. Accordingly, a ``double gap'' for the winding number can be established, thus producing solutions \emph{in pairs}. If no asymptotic hypothesis at zero is required (keeping, however, the validity of the local sign assumption $g(u) u > 0$ for $u \neq 0$), then $2k$ solutions are still preserved for $\lambda > \lambda^*_k$, thanks to the gap between intermediate and large solutions. 
Of course, depending on the growth condition at zero, a more precise picture may be displayed. For instance, combining the arguments of the present paper with the ones in \cite{CapDamZan01}, it is possible to find, for every $\lambda > 0$  (and assuming $q(\vert x \vert)> 0$), infinitely many nodal radial solutions of the equation 
$$
\textnormal{div} \, \left( \frac{\nabla u}{\sqrt{1-\vert \nabla u \vert^2}}\right) +  \lambda q(\vert x \vert) \vert u \vert^{p-1} u = 0, \qquad 0 < p < 1,
$$
having any arbitrary integer number of nodal regions. For completeness, we also mention \cite{BosGar13} for the case when $g(u)$ is linear at zero, in the context of the one-dimensional periodic problem.  
\end{remark}
\bigbreak
\noindent
\textbf{Acknowledgement.} The authors were supported by the INDAM-GNAMPA Project 2016 ``Problemi
differenziali non lineari: esistenza, molteplicit\`a e propriet\`a qualitative delle
soluzioni''. The first author also acknowledges the support of the project
ERC Advanced Grant 2013 n. 339958 ``Complex Patterns for Strongly Interacting
Dynamical Systems - COMPAT''.

\end{document}